\documentclass[12pt]{amsart}
\usepackage[margin=1in]{geometry}
\usepackage{amsfonts,amssymb,latexsym,amsmath,amscd}
\usepackage{dbnsymb}
\usepackage{times}

\renewcommand{\AA}{\mathbb{A}}

\newcommand{\CC}{\mathbb{C}}

\newcommand{\JJ}{\mathbf{P}}
\newcommand{\JJbar}{\overline{\mathbf{P}}}
\newcommand{\NN}{\mathbb{N}}

\newcommand{\ZZ}{\mathbb{Z}}
\newcommand{\QQ}{\mathbb{Q}}
\newcommand{\PP}{\mathbb{P}}

\newcommand{\oO}{\mathcal{O}}

\newtheorem{lemma}{Lemma}
\newtheorem{conjecture}[lemma]{Conjecture}
\newtheorem{corollary}[lemma]{Corollary}
\newtheorem{theorem}[lemma]{Theorem}
\newtheorem{definition}[lemma]{Definition}
\newtheorem{example}[lemma]{Example}

\newtheorem{proposition}[lemma]{Proposition}

\DeclareMathOperator{\res}{res}
\DeclareMathOperator{\Hom}{Hom}
\DeclareMathOperator{\Ext}{Ext}
\newcommand*{\hHom}{\mathcal{H}\!\mathit{om}}
\newcommand*{\eExt}{\mathcal{E}\!\mathit{xt}}

\newcommand{\Hilbert}[1]{#1^{[\star\,]}}

\title{The Hilbert scheme of a plane curve singularity\\ 
and the HOMFLY polynomial of its link}
\date{}

\author{Alexei Oblomkov \and Vivek Shende}

\begin{document}

\begin{abstract}
The intersection of a complex plane curve with a small three-sphere
surrounding one of its singularities is a non-trivial link.  
The refined punctual Hilbert schemes of the singularity 
parameterize subschemes supported at the singular point of fixed
length and whose defining ideals have a fixed
 number of generators.  
We conjecture that the generating function of Euler characteristics 
of refined punctual Hilbert schemes is the HOMFLY polynomial of the 
link.  The conjecture is verified  for irreducible singularities 
$y^k\! = x^n$, whose links are the $(k,n)$ torus knots, and 
 for the singularity 
$y^4 = x^7 - x^6 + 4x^5y + 2x^3 y^2$, whose link is 
the (2,13) cable of the trefoil.
\end{abstract}

\maketitle

\section{Introduction}

Let $C$ be an integral complex curve with at worst 
locally planar singularities.  We write $C^{[l]}$ for its Hilbert
scheme of points \cite{G}, 
the moduli space of closed subschemes of $C$ of dimension
zero and length $l$.   We are interested in 
the Euler characteristics\footnote{
For a  topological space $Y$, we write
$\chi(Y) = \sum (-1)^i \dim \mathrm{H}^i(Y,\QQ)$; for specificity, we
take singular cohomology.  When $Y$ is a complex algebraic variety,
this is equal to the analogous $\chi_c(Y)$ formed from compactly 
supported cohomology.  If $Z$ a closed subvariety of $Y$, 
then $\chi(Y \setminus Z) = \chi(Y) - \chi(Z)$.   
If $E \to B$ is a topologically locally trivial 
fibration with fibre $F$, then
$\chi(E) = \chi(F) \chi(B)$.  In particular, if
$Y$ admits the action of an algebraic torus $T = (\CC^*)^n$, 
then $\chi(Y) = \#Y^T$. 
We use integral notation for weighted
Euler characteristics: if $\phi:Y \to A$ is a constructible function
taking values in an abelian group, then 
$\int_Y \phi\, d\chi := \sum a \cdot \chi(\phi^{-1}(a))$.
} 
of these spaces.
 For notational convenience, we write
$\Hilbert{C}:= 
\coprod {C}^{[l]}$, and separate the components by the function
$ l: \Hilbert{C}  \to  \NN$ counting the number of points.
In the case of a smooth curve $C_{\mathrm{sm}}$, not necessarily complete,
 the following formula is well known:
\begin{equation} \label{eq:eulofhilb:curve}
 \int\limits_{\Hilbert{C_{\mathrm{sm}}}} \!\! q^{2l}\, d\chi = 
 \left(\frac{1}{1-q^2}\right)^{\chi(C_{\mathrm{sm}})}
\end{equation}

For any point $p \in C$, denote by 
$C^{[l]}_p$ the closed subvariety of $C^{[l]}$ whose closed points
parameterize subschemes of $C$ 
that are set-theoretically supported at $p$.  We collect
these into $\Hilbert{C}_p = \coprod C^{[l]}_p$ which we equip with a 
function $l$ giving the length.  Now let $p_i$ be the singular points
of $C$, and take $C_{\mathrm{sm}}:=C \setminus \coprod p_i$. Stratifying 
$\Hilbert{C}$ by the length supported at the 
singular points, we find:
\begin{equation*}
  \int\limits_{\Hilbert{C}} \! q^{2l} \, d\chi \,\,\, = 
  \int\limits_{\Hilbert{{C_{\mathrm{sm}}}}} \! q^{2l} \, d\chi 
  \times \prod_i \int\limits_{\Hilbert{{C}_{p_i}}} \! q^{2l} \, d\chi \,\,\, =\,\,\, 
  \left(\frac{1}{1-q^2}\right)^{\chi(C_\mathrm{sm})}
  \prod_{i} \int\limits_{\Hilbert{{C}_{p_i}}} \! q^{2l} \, d\chi
\end{equation*}

The contributions from the singularities can be rather complicated.  
We will describe them conjecturally in terms of the topology of $C$ 
near its singular points.  Embed $C$ in a surface.  
Intersecting $C$ with a small 3-sphere 
around $p \in C$ yields the {\em link of $C$ at $p$}: 
a collection of oriented circles in $S^3$, 
one for each analytic branch of $C$ at $p$.   
For example, the curve $y^k = x^n$
intersects a sphere around the origin in the right-handed $(k,n)$ torus link.  
Milnor \cite{M} has studied links of hypersurface singularities, 
and shows in particular that the complement in the sphere of
any such link is a smooth fibre bundle over the circle.  
Eisenbud and Neumann \cite{EN} describe how to pass from
the combinatorics of the Puiseux data of the
singularity to a presentation of the link.  Campillo, Delgado, and
Gusein-Zade have shown that the Alexander polynomial of the link carries
data equivalent to the Hilbert series of a certain filtration on the
ring of functions on the singularity \cite{CDG}.

Denote by $\overline{\mathbf{P}}(L)$ the HOMFLY polynomial of an 
oriented link $L \subset S^3$. 
It is an element of  $\ZZ[a^{\pm 1},(q-q^{-1})^{\pm 1}]$, and may be computed
from the relations
\begin{eqnarray}
  \label{eq:skein1}
  a \, \JJbar(\undercrossing) -  a^{-1} \, 
  \JJbar(\overcrossing) & = & (q - q^{-1})  
  \, \JJbar(\smoothing) \\
  a - a^{-1} & = & (q-q^{-1})\JJbar(\bigcirc)
\end{eqnarray}
The terms of Equation (\ref{eq:skein1}) should be interpreted
as the HOMFLY polynomials of three link diagrams which
are identical away from a small neighborhood, and are as depicted within it.  
It is not obvious
that these relations define a function on diagrams, 
let alone knots, but it is true \cite{HOMFLY}.  
It is often convenient to adopt the alternative normalization
$\JJ(L) := \JJbar(L)/\JJbar(\bigcirc)$.
The topological meaning of the HOMFLY polynomial
is unknown.  

\begin{conjecture} 
  \label{conj:justqs}
  Let $C$ be a curve in a smooth surface, $C$ itself 
  smooth away from points $p_i$.  Let $L_i$ be the link of $C$ at $p_i$, and
  let $\mu_i$ be the Milnor number of the singularity at $p_i$.  Then,
  \begin{equation*}
    \int\limits_{\Hilbert{C}} \! q^{2l} \, d\chi \,\,\, = (1-q^2)^{-\chi(C)}
    \prod_i \left[ (q/a)^{\mu_i} \JJ(L_i) \right]_{a=0}
  \end{equation*}
\end{conjecture}

\noindent {\bf Aside.} 
We pause to mention the role of the series on the left hand side of the 
conjecture in the curve-counting theory of Pandharipande and Thomas.
If $\tilde{g}$ and $g$ 
are the geometric and arithmetic genera of $C$, then there are 
integers $n_h$ for $\tilde{g} \le h \le g$ such that \cite{PT3}
\begin{equation}
  \int\limits_{\Hilbert{C}} \! q^{2l} \, d\chi \,\,\, = \,\,\,
  q^{2g-2} \sum_{h=\tilde{g}}^g n_h(C) (q^{-1} - q)^{2h-2}
\end{equation}
The second author has shown \cite{S} that
these integers are positive, and record the degree 
of the subvariety of the versal deformation of $C$ 
parameterizing deformations with $g-h$ nodes.  
In the simplest cases, for instance if $C$ is a curve in a Fano surface of an
irreducible homology class $\gamma$, the $n_h(C)$ give the local
contribution of $C$ to the Gopakumar-Vafa \cite{GV} invariant
measuring the number of genus $h$ curves of class $\gamma$ in the local
Calabi-Yau formed by the canonical bundle of the surface. 
We are unaware of any physical interpretation of such local contributions:
roughly speaking, 
the Gopakumar-Vafa
numbers arise from identifying a certain space of BPS states
in M-theory compactified on the Calabi-Yau of interest with the cohomology
of the moduli space of rank one sheaves supported on {\em any} curve of class
$\gamma$.  

There is also a connection to quantum field theory on the
other side of the conjecture: Witten \cite{W} has explained that
the HOMFLY polynomial may be understood as collecting, over $k$ and $N$, the
expectation value of the Wilson loop around the link
in the $\mathrm{SU}(N)$ Chern-Simons theory on the three-sphere at
level $k$.  One might hope to give a ``physics proof'' 
of the above conjecture via 
a theory which, on the one hand gives rise to local Gopakumar-Vafa numbers, 
and, on the other, specializes
near any point of $C$ to Chern-Simons 
theory on the bounding 3-sphere.

\vspace{2mm}

We turn to the question of 
how the HOMFLY polynomial, rather than merely its $a \to 0$ limit, 
may be recovered from the Hilbert scheme.  We define an incidence
variety:
\[C_p^{[l]} \times C_p^{[l+m]} \supset 
C_p^{[l,\,l + m]} := \{(I,J)\, | \, 
I \supset J \supset I\oO_C(-p)\}\]
\begin{conjecture} 
  \label{conj:homfly}
    Let $p \in C$ be a point on a locally planar curve, let $L_{C,p}$ 
  be the link of $C$ at $p$, and let $\mu$ be the Milnor number
  of the singularity at $p$.  Then,
  \begin{equation*}
    \JJbar(L_{C,p})  =   (a/q)^{\mu-1} \sum_{l,m}
    q^{2l} (-a^2)^m \chi(C_p^{[l,\,l + m]})
  \end{equation*}
\end{conjecture}

We may integrate out the incidence variety.  
By Nakayama's lemma, the possible $J$ for given $I$ are
parameterized by a Grassmannian inside 
$I / I\oO_C(-p) = I \otimes \CC_p$.  Define
\begin{eqnarray*}
   m:\Hilbert{C_p} & \to & \NN \\
    \, I & \mapsto & \dim I \otimes \CC_p
\end{eqnarray*}
Equivalently, $m(I)$ is the minimal number
of generators of the ideal $I \otimes \hat{\oO}_{C,p}$. 
The contribution of $I$ to the RHS of Conjecture \ref{conj:homfly} 
is $(1-a^2)^{m(I)}$, so we have the following equivalent formulation
(which we have given in terms of the normalized HOMFLY polynomial):

\vspace{2mm} 
{\bf Conjecture 2'.}
  {\em Let $p \in C$ be a point on a locally planar curve, let $L_{C,p}$ 
  be the link of $C$ at $p$, and let $\mu$ be the Milnor number
  of the singularity at $p$.  Then,}
  \begin{equation*}
    \JJ(L_{C,p})  =   (a/q)^{\mu} (1-q^2) \!\! \int\limits_{\Hilbert{{C_p}}}\!\!
    q^{2l} (1-a^2)^{m-1} \, d \chi
  \end{equation*}
\vspace{2mm}

 Conjecture \ref{conj:homfly} implies Conjecture \ref{conj:justqs} by a 
stratification argument.  The remainder of the article presents
evidence for the conjectures, which we briefly itemize below.

\begin{itemize}
\item 
  Setting $a = -1$ in Conjecture \ref{conj:homfly}' leaves the formula
  \begin{equation*} 
    \nabla(L_{C,p}) \,\,\,  = \,\,\,  (-q)^{-\mu} (1-q^2) \!\!\!\!\!
    \int\limits_{\oO_{C,p}/\oO_{C,p}^*}\!\!\!\!\!
    q^{2l} \, d \chi
  \end{equation*}
  where $\nabla$ is the Alexander polynomial in a suitable normalization, and
  $\oO_{C,p}/\oO_{C,p}^*$ parameterizes functions up to multiplication by
  invertible functions, or in other words, ideals with one generator.  
  After some technical rearrangements, this formula follows from
  a  theorem of Campillo, Delgado, and Gusein-Zade. 
  Details appear in Section \ref{sec:CDG}.
\item 
  The skein relation exhibits the invariance the HOMFLY polynomial 
  under the transformation $q \to -q^{-1}$; in fact,
  $\JJ \in \ZZ[(q-q^{-1})^{\pm 1}, a^{\pm 1}]$.  
  It is not evident
  from their expressions that our integrals enjoy the same property;
  nonetheless this was verified by Pandharipande and Thomas for the
  integral in Conjecture 1, and we explain in Section \ref{sec:symmetry}
  how to extend their methods to the integral in Conjecture 2. 
\item
  Singularities of the form $y^k = x^n$ carry a torus action which
  lifts to the Hilbert scheme.  In the case $\gcd(k,n) = 1$, the fixed
  points are isolated, and we count them in Section \ref{sec:torus} to
  calculate the integral in Conjecture 2.  
  Jones has calculated the HOMFLY polynomial of the corresponding
  $(k,n)$-torus knot, and the formulae match.
\item
  Section \ref{sec:Hilbert} describes, in the case of unibranch singularities,
  a stratification of the Hilbert scheme of points via the semigroup of
  the singularity.  This is closely related to Piontkowski's work on 
  computing the cohomology of compactified Jacobians \cite{Pi}.  
  In Section \ref{sec:4613} 
  we compute explicitly the strata of the Hilbert scheme 
  of the singularity with complete local ring $\CC[[t^4, t^6 + t^7]]$,
  and verify that the generating function of its Euler characteristics
  matches the HOMFLY polynomial of the (2,13) cable of the right-handed
  trefoil knot.
\end{itemize}

\noindent {\bf Remark.} It is natural to promote  
the left 
hand side of Conjecture \ref{conj:homfly} to an integral against the 
weight polynomial;  
the sum now computes what may be regarded
as the homology of a bigraded space.  On the other hand, the HOMFLY polynomial 
is known to arise as the Euler characteristic of the cohomology of a bigraded 
complex  \cite{KR}.  We will state a homological version of 
Conjecture \ref{conj:homfly} in a subsequent article \cite{ORS}.

\vspace{2mm} \noindent {\bf Acknowledgements.}  
We are grateful to Rahul Pandharipande for suggesting this area of 
study -- during a class for which the notes may be found on his website
\cite{P} -- and 
for advice throughout the project.  We have also enjoyed
 discussions with Margaret Doig, Eduardo Esteves, Paul Hacking,
Jesse Kass, Jacob Rasmussen, Giulia Sacc\`a, Sucharit Sarkar, 
Richard Thomas, and Kevin
Wilson, and  thank the referees for many helpful comments.
A.O. was partially supported by NSF grant DMS-0111298.

\section{Smooth points, nodes, and cusps} \label{sec:sanity}

We illustrate the conjecture with some elementary examples.  
Denote by $C_{2,n}$ the formal germ at the origin of the curve cut out 
by $y^2 = x^n$, and by $\oO_{2,n}$ its ring of functions.  
The link of this singularity is the right-handed
$(2,n)$ torus link $T_{2,n}$.  The first few of these:

\[T_{2,0} = \bigcirc \bigcirc \,\,\,\,\,\,\,\, 
T_{2,1} = \bigcirc \,\,\,\,\,\,\,\,
T_{2,2} = \HopfLink \,\,\,\,\,\,\,\,
T_{2,3} = \righttrefoil \]

Computing $\JJ(T_{2,n})$ is an elementary
exercise in the skein relation: smoothing a crossing yields 
$T_{2,n-1}$ and switching a crossing gives $T_{2,n-2}$.  This yields the recurrence
\[ 
\JJ(T_{2,n}) =  - a (q-q^{-1}) \JJ(T_{2,n-1}) + a^2 \JJ(T_{2,n-2})
\] 
$T_{2,1}$ is the unknot, and $T_{2,0}$ is two unlinked circles.  It is
immediate from the skein relation that 
the HOMFLY polynomial of $n$ unlinked
circles is $((a-a^{-1})/(q-q^{-1}))^{n-1}$.  Thus:
\begin{eqnarray}
  \JJ(T_{2,0})  & = &  \frac{a-a^{-1}}{q-q^{-1}} \\
  \JJ(T_{2,1}) & = & 1 \\
  \JJ(T_{2,2}) & = & - a(q-q^{-1})  + \frac{a^3 - a}{q-q^{-1}} \\
  \JJ(T_{2,3}) & = &   a^2 q^2 + a^2 q^{-2} - a^4
\end{eqnarray}

We now compute the integral of Conjecture \ref{conj:homfly} for
$n = 1, 2, 3$. 

\begin{example}  As $y^2 = x$ is smooth at the origin, the 
Milnor number is $\mu = 0$.  The ring $\oO_{2,1} =  \CC[[t]]$
has ideals $(t^i)$ for $i \in \NN$.  Then the conjecture
asserts
\[1 = (a/q)^0 (1-q^2) \sum_{i=0}^\infty q^{2i}\]
\end{example}

\begin{example}
  At the origin, $y^2 = x^2$ has a node, so the Milnor number
  is $\mu = 1$.  We parameterize by $t_1 = x-y$ and $t_2 = x+y$ to write
  $\oO_{2,2} = \CC[[x,y]]/(x^2 - y^2) = \CC[[t_1, t_2]]/(t_1 t_2)$.  
  The finite colength ideals of this ring are:
  \[ 
  \begin{array}{ll}
    (1) &  \\
    (t_1^k + \lambda t_2^{i-k}) & \text{for $1 \le k < i$ 
      and $\lambda \in \CC^*$} \\
    (t_1^k, t_2^{i-k+1}) & \text{for $1 \le k \le i$} \\
    \end{array}
  \]
  In each case the variable $i$ gives the colength of the ideal. 
  Each component of the space of ideals of the second type is $\CC^*$ and
  thus has Euler characteristic zero.  Thus the conjecture asserts
  \[- a(q-q^{-1})  + \frac{a^3 - a}{q-q^{-1}} = (a/q)^1 (1-q^2) 
  \left ( 1 + (1-a^2) \sum_{i=1}^\infty i q^{2i}  \right)\]
\end{example}

\begin{example}
  At the origin, $y^2 = x^3$ has a cusp, so the Milnor number is
  $\mu = 2$.  The ideals of the ring $\oO_{2,3} = \CC[[t^2, t^3]]$ are:
  \[
  \begin{array}{ll}
    (1) & \\
    (t^i + \lambda t^{i+1}) & \text{for $i \ge 2$ and all $\lambda \in \CC$} \\
    (t^{i+1}, t^{i+2}) & \text{for $i \ge 1$} \\
  \end{array}
  \]
  Each component of the space of ideals of the second type is $\CC$ and
  thus has Euler characteristic 1. Thus the conjecture asserts
  \[ a^2 q^2 + a^2 q^{-2} - a^4 = 
  (a/q)^{2} (1-q^2) \left( 1 + \sum_{i=2}^\infty q^{2i} 
    + (1-a^2) \sum_{i=1}^\infty q^{2i}  \right)
  \]
\end{example}

\section{Specialization to the Alexander polynomial}
\label{sec:CDG}

In this Section, we show Conjecture \ref{conj:homfly} holds in the limit $a = -1$:  

\begin{proposition} \label{prop:alex}
Let
$C$ be the germ of a plane curve singularity, and $L_C$ its link.  Then
\[\JJ(L_C)|_{a=-1} = \lim_{a\to -1}
  (a/q)^{\mu} (1-q^2) \!\! \int\limits_{\Hilbert{{C}}}\!\!
    q^{2l} (1-a^2)^{m-1} \, d \chi
\]
\end{proposition}

Both sides of the above equality have simpler expressions.  
The left hand side is the Alexander-Conway polynomial, denoted $\nabla(L_C)$,
as can be seen by specializing the skein relations:
\begin{eqnarray*}
  \nabla(\overcrossing) -  \nabla(\undercrossing) & = & (q - q^{-1})  
  \nabla(\smoothing) \\
  \nabla(\bigcirc) & = & 1
\end{eqnarray*}

Since the integrand on the right hand side vanishes unless $m = 1$,  
we integrate only over the principal ideals of finite colength.
Observe that $\dim \oO_C/f\oO_C < \infty$
if and only if $f$ is regular, i.e., neither zero nor a zero divisor; we write
$\oO_C'$ for the set of regular elements.
Thus finite colength ideals
are parameterized by $\oO_C'/ \oO_C^*$ and 
Proposition \ref{prop:alex} is equivalent to:
\begin{equation}
\label{eqn:Alexander}
  \nabla(L_C) = 
  (1-q^2)(-q)^{-\mu} \!\! \int \limits_{\oO_C'/\oO_C^*} \! q^{2l} \, d\chi
\end{equation}
where at a point of $\oO'_C/\oO_C^*$ represented by a function $f$,
the function $l$ takes the value $\dim \oO_C/f\oO_C$.  

We will derive Equation (\ref{eqn:Alexander}) from a theorem of 
Campillo, Delgado, and Gusein-Zade which recovers the
Alexander polynomial from the {\em extended semigroup} \cite{CDG}.  
Fix a normalization
$n:\oO_C \hookrightarrow \bigoplus_{i=1}^b \CC[[z_i]]$, where $b$ is the number
of analytic local branches of $C$.   Denoting by $\nu_i$ the valuation on
$\CC[[z_i]]$ which gives the degree of the lowest order term.
We have by restriction a map 
$\nu = (\nu_1,\ldots, \nu_b): \oO_C' \to \NN^b$.  Its image $\Gamma$ is 
called the 
{\em semigroup} of the curve singularity.\footnote{
This notion is classical
at least when $b=1$, for a thorough discussion we refer to \cite{ZT}. 
}  It depends on the normalization
$n$ only up to reordering of the $z_i$.  
Recording both the degree and the coefficient of the lowest order
term gives 
$\overline{\nu}_i: \CC[[z_i]]\setminus 0 \to \NN \times \CC^*$, and 
by restriction
$\overline{\nu}: \oO_C' \to (\NN \times \CC^*)^b$.  The image
$\overline{\Gamma}$ is called the extended semigroup.  We write
$\PP \overline{\Gamma}$ for the quotient by the diagonal $\CC^*$; note
that in the case $b=1$ the composition
$\Gamma \hookrightarrow 
\overline{\Gamma} \to \PP \overline{\Gamma}$ is a bijection.
Evidently 
$\nu_i:\oO_C'\to \NN$ descends to a function on $\PP \overline{\Gamma}$.

\begin{theorem} 
\label{thm:CDG}
  (Campillo--Delgado--Gusein-Zade \cite{CDG}) Let $C$ be the germ of
  a plane curve singularity with $b$ branches, and let $L$ be its link.
  If $b > 1$, the multivariate
  Alexander polynomial is given by 
  \begin{equation}\label{eq:Alexander}
    \Delta_L(t_1,\ldots,t_b) = 
    \int\limits_{\PP\overline{\Gamma}} \!
    t_1^{\nu_1}\cdots t_b^{\nu_b} \, d \chi
  \end{equation}
  For any $b$, the one-variable Alexander polynomial is given by
  \begin{equation}\label{eqn:CDGonevar}
    \Delta_L(t) = 
    (1-t) \int\limits_{\PP \overline{\Gamma}} \!
    t^{\sum \nu_i} \, d \chi
  \end{equation}  
  The Alexander polynomials have been normalized so 
  $\Delta_L(t_1,\ldots,t_b) \in 1 + (t_1,\ldots,t_b) \ZZ[t_1,\ldots,t_b]$ and
  $\Delta_L(t) \in 1 + t \ZZ[t]$. 
\end{theorem}

\noindent{\bf Remark.}  The statement about the one-variable Alexander
polynomial in case $b>1$ is not given explicitly in \cite{CDG} but
follows immediately by specializing.  Indeed, 
$(1-t) \Delta_L(t,\ldots,t) = \Delta_L(t)$, see for instance \cite[Prop. 7.3.14]{Ka}.  

\begin{lemma}
  Let $C$ be the germ of a plane curve singularity; let $L$ be its link and
  $\mu$ its Milnor number.  Normalize its Alexander
  polynomial by requiring
  $\Delta_L(t)\in 1 + t\ZZ[t]$ and its Alexander-Conway polynomial
  $\nabla_L(q)$ by the skein relation given above.  Then
  $\Delta_L(q^2) = (- q)^{\mu} \nabla_L(q)$.
\end{lemma}
\begin{proof}
  It is well
  known that $\Delta_L(q^2) = \pm q^{n} \nabla_L(q)$.  Since 
  $\nabla_L(q) = \nabla_L(-1/q)$, the integer $n$ must
  be the degree of the Alexander polynomial, which Milnor shows to be 
  $\mu$ \cite{M}.  
  To resolve the sign ambiguity, recall that the  link of a plane curve 
  singularity  may be realized as the closure
  of a braid in which only positive crossings $(\overcrossing)$ appear.  
  Van Buskirk has shown that such
  ``positive links have positive Conway polynomials \cite{vB},''
  meaning
  $\nabla_L(q) = \sum n_i (q-q^{-1})^i$ for $n_i \in \NN$.
  \footnote{In van Buskirk's paper, much more precise
    conditions on the $n_i$ are given.}
\end{proof}

We rewrite Equation (\ref{eqn:CDGonevar}) as
\begin{equation} \label{eqn:CDGonevar2}
 \nabla_L(q) = (1-q^2)(-q)^{-\mu} \!
\int\limits_{\PP \overline{\Gamma}}\!
q^{2\sum \nu_i} \, d \chi
\end{equation}

To prove Proposition \ref{prop:alex}, it remains
to match the integrals in Equations
(\ref{eqn:Alexander}) and (\ref{eqn:CDGonevar2}).  
There are two apparent differences: the exponent on $q$ in
the integrand, and the space over which we integrate.  
As a consequence of the following Lemma, applied with $A= \oO_C$ and 
$B$ its normalization, in fact
$\sum \nu_i(f) = \dim \oO_C/ f \oO_C$ as functions on $\oO_C'$, 
hence the integrands agree.

\begin{lemma} \label{lem:lengths}
  Let $A \hookrightarrow B$ be rings and let $f \in A$ be a non zero divisor
  in both $A$ and $B$.
  If $B / A$,  $A / fA$, and $B / fB$ 
  have finite length as $A$-modules, then
  $A/fA$ and $B / fB$ have equal length.
\end{lemma}
\begin{proof}
  Consider the diagram 
\[
\begin{CD}
      0 @>>>  fA @>>>  A @>>>  A/fA @>>>  0 \\
      @. @VVV @VVV @VVV @. \\
      0 @>>>  fB @>>>  B @>>>  B/fB @>>>  0
\end{CD}
\]
The Snake Lemma provides  the long exact sequence
\begin{eqnarray*}
0 & \to & \ker(fA \to fB) \to \ker(A \to B) \to
\ker(A/fA \to B/fB) \\
& \to & fB/fA \to B/A \to (B/fB)/(A/fA) \to 0
\end{eqnarray*}
Note that the first two modules in each line are abstractly 
isomorphic (indeed, the ones on the first line are zero).  Since
all these modules have finite length, the alternating sum of the
lengths is zero, and hence $\ker(A/fA \to B/fB)$ and 
$(B/fB)/(A/fA)$ have the same length.  But then $A/fA$ and $B/fB$ have
the same length. 
\end{proof}

\noindent{\bf Remark.}  The natural map 
$A/fA \to B/fB$ is not generally an isomorphism.  For example, 
if $A = \CC[[t^2, t^3]] \subset \CC[[t]] = B$ and $f = t^2$, then
$\ker(A/fA \to B/fB)$ is one dimensional spanned by the class of
$t^3$, and  $\mathrm{coker}(A/fA \to B/fB)$ is one dimensional spanned
by the class of $t$. 

Now we compare the spaces $\PP \overline{\Gamma}$ and $\oO_C'/\oO_C^*$.  
Let $\oO_C^{1*}$ denote the monic invertible functions, and factor 
$\overline{\nu}$ as 
$\oO_C'\xrightarrow{\pi} \oO_C'/\oO_C^{1*} \to \overline{\Gamma}$. 
We study the fibres.  If
$\overline{\nu}(f) = \overline{\nu}(g)$ then 
$\overline{\nu}(f) = \overline{\nu}((1-\lambda)f+\lambda g)$ for any 
$\lambda \in \CC$.   If in addition $f$ and $g$ generate
the same ideal, then certainly we have
an inclusion $((1-\lambda)f + \lambda g)\oO_C \subset f \oO_C$.  But
by Lemma \ref{lem:lengths}, these ideals have the same colength
and thus are equal.  We conclude that the fibres of 
$\overline{\nu}$ and $\pi$ are
(infinite dimensional) affine spaces.  All maps descend
to the quotient by $\CC^*$ without changing fibres, and if
we are willing to take the Euler characteristic of infinite dimensional
spaces,
\[ 
\int\limits_{\PP \overline{\Gamma}}\!
q^{2\sum \nu_i} \, d \chi
= \int \limits_{\PP \oO_C'}\! q^{2 \sum \nu_i} \, d\chi = 
\!\!\int \limits_{\oO_C'/\oO_C^*} \!\! q^{2 \sum \nu_i} \, d\chi
\]
The Euler numbers of infinite dimensional spaces can 
made sense of as in \cite{CDG2}.  They can also be avoided.  
At any 
fixed degree $\sum \nu_i = d$, there exists some $N \gg 0$
(twice the $\delta$ invariant will do) such that
any principal ideal of colength $d$ 
can be generated by some $f$ whose expansion as an element of
$\bigoplus \CC[[z_i]]$ has no terms of degree larger than $d+N$. 
We may use the finite dimensional space of such functions in place
of $\oO_C'$ to conclude the above equality at degree $d$. 

This completes the proof of Proposition \ref{prop:alex} $\blacksquare$


\section{The genus expansion}
\label{sec:symmetry}

The HOMFLY polynomial lies in $\ZZ[a^{\pm 1}, (q-q^{-1})^{\pm 1}]$; in 
particular,  
the skein relation defining it is manifestly invariant
under the involution $q \to -1/q$.   We will show in this Section that the
same properties hold for the quantity on the right hand side of Conjecture
\ref{conj:homfly}.  In the $a \to 0$ limit of Conjecture \ref{conj:justqs},
this is proven by Pandharipande
and Thomas \cite[Appendix B]{PT3}.
Their approach ultimately rests on Serre duality and the Abel-Jacobi map,
and works without modification for the series in Conjecture
\ref{conj:homfly}' once note is taken of 
the following fact from commutative algebra \cite{E}:

\begin{lemma} \label{lem:dualgens}
  Let $C$ be a reduced integral locally planar curve, and let
  $\mathcal{F}$ be a torsion free sheaf over $C$.  Then for any point
  $p \in C$ and any line bundle $\mathcal{L}$, 
  \[\dim_\CC \mathcal{F} \otimes \CC_p = 
  \dim_\CC \mathcal{H}om(\mathcal{F},\mathcal{L}) \otimes \CC_p\]
\end{lemma}
\begin{proof}
  Taking
   $A = \oO_{C,p}$ and $M = \mathcal{F}_p$ it suffices to show
  $\dim M \otimes_A \CC = 
  \dim \mathrm{Hom}_A(M,A) \otimes_A \CC$, where
  $\CC$ is the residue field $A / \mathfrak{m}_p A$. 

  Locally $C$ embeds in a surface $S$; let $R = \oO_{S,p}$ and
  $A = R/fR$.  As
  $\mathcal{F}$ was torsion free, any element of $A$ lifts to
  an element of $R$ which is regular on $M$. 
  Thus $M$ has depth at least $1$ over $R$.  As $R$ is regular,
  the Auslander-Buchsbaum theorem guarantees that
  \[\mathrm{proj.}\dim(M) = \mathrm{depth}(R) - 
  \mathrm{depth}(M) \le 2 - 1
  = 1\]
  Since $M$ is not a free $R$-module, we have 
  $\mathrm{proj.}\dim(M) = 1$.  
  Fix a minimal two-term resolution of $M$.  Since $M$ is rank
  zero as an $R$-module, the two terms have the same 
  rank, namely $m := \dim M \otimes_R \CC = \dim M \otimes_A \CC$.
  On the other hand, 
  \[ 0 \to R^m \to R^m \to M \to 0\]
  gives rise to 
  \[ 0 = \Hom_R(M,R) \to R^m \to R^m \to \Ext^1_R(M,R) \to 0\]
  and thus
  $\CC^m \twoheadrightarrow \Ext^1_R(M,R) \otimes \CC = 
  \Hom_A(M,A) \otimes \CC$.  We conclude 
  $\dim \Hom_A(M,A) \otimes \CC \le \dim M \otimes \CC$. 
  As $A$ is Gorenstein, 
  we may dualize again for the reverse inequality.
\end{proof}

We pass to a complete curve.  
\begin{lemma} \label{lem:stratification}
Let $C$ be a rational curve smooth away 
from a single point $p$, where it has $b$ analytic local branches. 
Denote the length functions as $l(I) := \dim \oO_C/I$ and $l_p(I) := 
\dim \oO_{C,p} / I_p$, and the function measuring the number
of generators by $m(I) = \dim
I \otimes \CC_p$.  Then,
  \[  \int\limits_{\Hilbert{C}} \!\! q^{2l} (1-a^2)^{m-1} \, d \chi  = 
  (1-q^2)^{b - 2} \!\! \int\limits_{\Hilbert{C_p}}  
  \!\! q^{2l_p} (1-a^2)^{m-1} \, d \chi\]
\end{lemma}
\begin{proof}  Consider the map 
$\Hilbert{(C \setminus p)} \times \Hilbert{C_p} \to \Hilbert{C}$ which 
takes the  (disjoint) union of a scheme supported away from $p$ 
and a scheme supported at $p$.  The map is a bijection 
with constructible inverse.  Denote by
$l_{\overline{p}}: \Hilbert{(C\setminus p)} \to \NN$ the length function
on this space. For
$X$ supported at $p$, $Y$ supported away from $p$, and $Z$ their union,
we have 
$l_p([X]) + l_{\overline{p}}([Y]) = l_C([Z])$.  By definition, 
$m([Z]) = m([X])$.  Thus we compute
\[
  \int\limits_{\Hilbert{C}} \!\! q^{2l_C} (1-a^2)^{m-1} \, d \chi \,\,\,\, =  
  \int\limits_{\Hilbert{C \setminus p}} \!\!\! q^{2l_{\overline{p}}} \, d \chi
  \,\,\, \times 
  \int\limits_{\Hilbert{C_p}} \!\! q^{2l_p} (1-a^2)^{m-1} \, d \chi
\]
By Equation (\ref{eq:eulofhilb:curve}), the first term in the latter
product is $(1-q^2)^{-\chi(C \setminus p)}$.
\end{proof}

A special case of the BPS calculus of Pandharipande and Thomas is the following:

\begin{proposition} \label{prop:bps}
  (Pandharipande and Thomas \cite{PT3}.)  Let $C$ be an
  integral Gorenstein curve
  of genus $g$, and
  let $\phi:\overline{\mathrm{Pic}}(C) \to A$ be a constructible function.  
  Assume that $\phi(F) = \phi(F \otimes L)$ for any line bundle $L$, and 
  moreover that $\phi(F) = \phi( \hHom(F,\omega_C))$.  Define
  $n_h(C,\phi)$ by the expansion   
  \[\int\limits_{\Hilbert{C}} q^{2 l} \cdot (\phi \circ AJ) \, d\chi = 
  \sum_{h=-\infty}^g q^{2g-2h} (1-q^2)^{2h-2} n_h(C,\phi)\]
  Then $n_h(C,\phi) = 0$ for all $h < 0$. 
\end{proposition}

\begin{theorem}  Let $C_p$ be the germ of a plane curve; let $\mu$ be
  the Milnor number. 
  Then there exist $n_h(a^2) \in \ZZ[a^2]$ such that
  \[ (a/q)^{\mu} (1-q^2) \!\! \int\limits_{\Hilbert{{C_p}}}\!\!
    q^{2l} (1-a^2)^{m-1} \, d \chi 
= a^\mu \sum_{h=0}^\delta n_h(a^2) (q^{-1}-q)^{2h+1-b} \]
  The degree in $a^2$ of $n_h(a^2)$ is at most one less than the
  multiplicity 
  of the singularity $C_p$. 
\end{theorem}
\begin{proof}
  As in Lemma \ref{lem:stratification}, let $C$ be a complete
  rational curve with a unique singularity at $p$ such that $C_p$ is
  the desired germ, and let $m:C^{[n]} \to \NN$ be defined on
  an ideal sheaf $I$ by 
  $m(I) = \dim_\CC I \otimes \CC_p$.  Evidently $m(I)$ depends
  only on the isomorphism class of $I$ as a sheaf, and moreover
  is unchanged upon tensoring $I$ with a line bundle.  
  By Lemma \ref{lem:dualgens}, we also have $m(I) = m(\hHom_{C}(I,\omega_C))$. 
  Therefore we may take $\phi = (1-a^2)^{m-1}$
  in Proposition \ref{prop:bps}.  To return to the situation for 
  the germ $C_p$, we apply Lemma \ref{lem:stratification} and recall
  $\mu = 2\delta + 1 - b$ \cite{M}.  

  The statement giving the degree of $n_h(a^2)$ amounts to the
  fact that the minimal number of generators
  of any ideal is bounded by the multiplicity of the singularity.
  In the unibranch case, this is elementary; in general, see
  \cite[Exercise 4.6.16]{BH}.  
\end{proof}


In the remainder of this Section, we give for completeness
an exposition of the proof of Proposition \ref{prop:bps}.  
We first recall some properties of torsion free sheaves on Gorenstein
curves, which we subsequently use without further comment.  We write
$F^*$ for $\hHom(F,\oO)$. 

\begin{lemma} \label{lem:hart}  (Hartshorne \cite{H}.)
  Let $C$ be an integral Gorenstein curve, $F$ a torsion free sheaf.
  Then higher extensions vanish, $\eExt^{\ge 1}(F,L) = 0$, and $F$ is
  reflexive, $F = (F^*)^*$.  Serre duality
  holds in the form $\mathrm{H}^i(F) = \mathrm{H}^{1-i}(F^* \otimes \omega_C)^*$.
  For $F$ rank one and torsion free, we define its 
  degree $d(F):=\chi(F) - \chi(\oO_C)$.   This satisfies 
  $d(F) = - d(F^*)$.  Moreover if $L$ is any line bundle,
  $d(F \otimes L) = d(F) + d(L)$. 
\end{lemma}

In great generality, Altman and Kleiman construct a
projective scheme $\overline{\mathrm{Pic}}(X)$ 
whose closed points parameterize rank one, torsion free sheaves on $X$ 
\cite{AK}.  We require here only the case where $X=  C$ is an
integral Gorenstein curve; some statements below are false 
for general curves.

The space $\overline{\mathrm{Pic}}(C)$ decomposes as a disjoint union 
$\coprod_d \overline{\mathrm{Pic}}_d(C)$ indexed by the degree of the sheaves.
An Abel-Jacobi map 
\begin{eqnarray*} 
  AJ: C^{[d]} & \to & \overline{\mathrm{Pic}}_d(C)\\
  Z & \mapsto & \hHom(I_Z,\oO_C)
\end{eqnarray*}
can be defined by sending a
subscheme to the dual of its ideal sheaf.  By Lemma
\ref{lem:hart},  $\hHom(\cdot,\oO_C)$ gives a bijection between
realizations of a sheaf $I$ as an ideal sheaf, i.e., inclusions
$I \hookrightarrow \oO_C$, and nonzero sections of $I^*$. 
Thus $AJ^{-1}(F) = \PP \mathrm{H}^0(F)$.  

Fix a line bundle $\oO(1)$ of degree $1$ on $C$.  For $F$ 
of degree zero we 
 consider 
\[H_F(q) := (1-q^2)^2 \sum q^{2d} \mathrm{h}^0(F(d))\]
Since $\phi$ is invariant under tensoring with line bundles, 
\[ 
  \int\limits_{\Hilbert{C}} \!\! q^{2l} (\phi \circ AJ) \, d \chi \,\,\,\, =  
  (1-q^2)^{-2} \int\limits_{\overline{\mathrm{Pic}}_0(C)} \!\! H \cdot 
   \phi \, d \chi  
\]

\begin{lemma} $H_F(q)$ is a polynomial of degree at most $4g$, and
  $H_F(q) = q^{4g} H_{F^* \otimes \omega \otimes \oO(2-2g)}(-1/q)$. 
\end{lemma}
\begin{proof}
  The function $h^0(F(d))$ is supported in $[0, \infty)$ and is equal to
  $d + 1 - g$ in $(2g-2,\infty)$.  Inside $[0,2g-2]$, it increases by
  either $0$ or $1$ at each step. 
  Let 
  \[n_{\pm}(F) = \{d \, | \, 2 \mathrm{h}^0(F(d-1)) = 
  \mathrm{h}^0(F(d)) + \mathrm{h}^0(F(d-2))  \pm 1\}\]  
  Note that $n_- \subset [0,2g]$ and $n_+ \subset [1,2g-1]$.  The expansion
  \[H_F(q) = \sum_{d \in n_-(F)}   q^{2d} - \sum_{d \in n_+(F)} q^{2d}\]
  establishes the polynomiality and degree of $H$.   Moreover, by 
  Serre duality and the Riemann-Roch formula, we have 
  $d \in n_{\pm}(F^* \otimes \omega \otimes \oO(2-2g)) \iff 2g-d \in n_{\pm}(F)$.
  This establishes the desired symmetry.
\end{proof}

Since $\phi$ is invariant under the involution
$F \mapsto F^* \otimes \omega \otimes \oO(2-2g)$ of $\overline{\mathrm{Pic}}_0$,
we can integrate the previous Lemma to find that 
$Z_\phi(q) := \int_{\overline{\mathrm{Pic}}_0(C)} H \cdot 
   \phi \, d \chi$ is a polynomial of degree at most $4g$ in $q$ and 
$Z_\phi(q) = q^{4g} Z_\phi(-1/q)$.  To finish the proof of Proposition
\ref{prop:bps}, it remains only to observe that 
$\{q^{2g-2h} (1-q^2)^{2h}\}$ span the space of such functions.

\section{The curve $y^k = x^n$} 
\label{sec:torus}

We consider now the singularity at the origin of 
$y^k = x^n$ for $k,n$ relatively prime.  
The complete local ring is $\CC[[t^k, t^n]]$, and 
the corresponding knot is the $(k,n)$ torus knot.  Jones 
has computed its HOMFLY polynomial.

\begin{theorem}[Jones]
\[
\JJ(T_{k,n}) =  \frac{(1-q^2)(a/q)^{(k - 1)(n - 1)}}{(1-q^{2k})(1-a^2)}
  \sum_{j=0}^{k-1} (-1)^j \frac{(q^2)^{j n + (k-1-j)(k-j)/2}}{[j]!\,[k-1-j]!}  
  \prod_{i=j+1-k}^j(q^{2i} - a^2)
\]
where $[0]! = 1$ and $[r]! = (1-q^{2r})[r-1]!$
\end{theorem}

The Milnor number of this singularity is $\mu = (k-1)(n-1)$. 
After rearranging the normalization factors, Conjecture \ref{conj:homfly}'
asserts:

\[ 
\!\!\!\! \int\limits_{\Hilbert{C_p}}\!\! q^{2l} (1-a^2)^m \, d \chi
=
\frac{1}{1-q^{2k}}
\sum_{j=0}^{k-1} (-1)^j \frac{(q^2)^{j n + (k-1-j)(k-j)/2}}{[j]!\,[k-1-j]!}  
  \prod_{i=j+1-k}^j(q^{2i} - a^2)
\]

To compute the left hand side we use a torus action.  
$\CC^*$ acts on $\CC[[t^k, t^n]]$ by scaling $t$. The action lifts
to the Hilbert scheme, and preserves the functions $l, m$ measuring
length and number of generators.  The integral with respect to Euler
characteristic of an $\CC^*$-equivariant function may always be computed
on the fixed locus of the $\CC^*$ action since the remainder of the space
will be fibred by $\CC^*$ and hence contribute zero to the Euler characteristic.
Diagonalizing the $\CC^*$ action on a fixed ideal will yield monomial
generators; conversely all the monomial ideals are fixed.  There
are countably many of these, and only finitely many with colength below
any given bound.  Therefore:

\[ 
\!\! \int\limits_{\Hilbert{C_p}}\!\! q^{2l} (1-a^2)^{m} \, d \chi\,\,\,\,\,
=\!\!
\sum_{J\,\, \mathrm{monomial}} \!\!\! q^{2 \dim_\CC \oO/J}(1-a^2)^{ m(J)}
\]

\begin{lemma} The monomial ideals are enumerated by the following
function.  \label{lem:res}
\[\sum_{J\,\, \mathrm{monomial}} \!\!\! q^{2 \dim_\CC \oO/J}(1-a^2)^{ m(J)}
 = \frac{1}{1-q^{2k}} \res_{\xi=0} \frac{1}{\xi^{n+1}}
\prod_{i=0}^{k-1} \left( 1 + (1-a^2)\frac{\xi q^{2i}}{1-\xi q^{2i}} \right)
\]
\end{lemma}
\begin{proof}
Monomial ideals of $\CC[[x,y]]$ can be matched 
with staircases \cite{Br, I}.  We proceed similarly. 
Consider the map
\begin{eqnarray*} 
\NN \times \{0,\ldots,k-1\} & \to & \mathrm{monomials} \in \oO \\
(\alpha,\beta) & \mapsto & t^{\alpha k + \beta n}
\end{eqnarray*}
It follows from the Chinese remainder theorem that this is a bijection. 
Monomial ideals are in 1-1 correspondence with sequences
$\phi = \phi_{k-1} \le \phi_{k-2} \le \ldots \le \phi_0 \le \phi_{k-1} + n$ 
via the 
correspondence 
\[\phi \leftrightarrow \{(\alpha, \beta) \, | \, \alpha > \phi_\beta \}\]
The number of generators of the ideal is the number of 
inequalities above which are strict.  The cardinality of the 
complement of the ideal is ``the number of boxes under the staircase,''
or $\sum \phi_i$. 

Regarding the $\phi_i$ as the lengths of the rows
of the staircase, the formula in the lemma enumerates the staircases 
column by column.  The leading
term $(1-q^{2k})^{-1}$
accounts for the leading columns of full height $k$.  The term $i$ in
the product corresponds to columns of height $i$.  The number of different
column heights is equal to the number of inequalities in $\phi$.  
The residue enforces the condition that 
there should be exactly $n$ columns of height less than $k$.
\end{proof}

\begin{example}  We give the staircase of $(t^{21}, t^{23}, t^{24}) \subset
\CC[[t^4, t^5]]$.  Bold numbers correspond to monomials in the ideal. 
\begin{center}
\begin{tabular}{|c|c|c|c|c|c|c|c}
  \hline
  {\tiny 15} & {\tiny 19} & {\bf 23} & {\bf  27}  & {\bf 31} & {\bf 35} 
  & {\bf 39} & {\bf 43} \\
  \hline
  {\tiny 10} & {\tiny 14} & {\tiny 18} & {\tiny 22} & 
  {\bf 26} & {\bf 30} & {\bf 34} & {\bf 38} \\
  \hline
  {\tiny 5}  & {\tiny 9}  & {\tiny 13} & {\tiny 17} & 
  {\bf 21} & {\bf 25} & {\bf 29} & {\bf 33} \\
  \hline
  {\tiny 0}  & {\tiny 4}  & {\tiny 8}  & {\tiny 12} & 
  {\tiny 16} & {\tiny 20} & {\bf 24} & {\bf 28} \\
  \hline 
\end{tabular}
\end{center}
\end{example}

\begin{example}  The following staircase does 
not correspond to any ideal of $\CC[[t^4, t^5]]$, because $28 = 23 + 5$.
This occurs because the staircase does not descend quickly enough.
\begin{center}
\begin{tabular}{|c|c|c|c|c|c|c|c|c}
  \hline
  {\tiny 15} & 
  {\tiny 19} & {\bf 23} & {\bf 27}  & {\bf 31} & {\bf 35} & {\bf 39} & {\bf 43} 
  & {\bf 47}\\
  \hline
  {\tiny 10} & 
  {\tiny 14} & {\tiny 18} & {\tiny 22} & 
  {\bf 26} & {\bf 30} & {\bf 34} & {\bf 38} & {\bf 42}\\
  \hline
  {\tiny 5}  & 
  {\tiny 9}  & {\tiny 13} & {\tiny 17} 
  & {\bf 21} & {\bf 25} & {\bf 29} & {\bf 33} & {\bf 37} \\
  \hline
  {\tiny 0}  & {\tiny 4}  & {\tiny 8}  & {\tiny 12} & 
  {\tiny 16} & {\tiny 20} & {\tiny 24} & {\tiny 28} & {\bf 32} \\
  \hline 
\end{tabular}
\end{center}
\end{example}

We evaluate the residue in Lemma \ref{lem:res} 
by summing over the other singularities of
the expression.  These occur precisely at $\xi = q^{-2j}$ for
$j = 0, 1, \ldots, k-1$.  Note that
\[\res_{z=1/w}\frac{1}{1-wz} = - \frac{1}{w}\]
in order to evaluate the residue:
\begin{equation}
\frac{1}{1-q^{2k}} \sum_{j = 0}^{k-1} q^{2(n+1) j} 
\left(\prod_{i = 0}^{j-1} \frac{1}{1 - q^{2(i-j)}}\right) q^{-2j} 
\left(\prod_{i=j+1}^{k-1} \frac{1}{1 - q^{2(i-j)}}\right) 
\left(\prod_{i=0}^{k-1} 1 -a^2 q^{2(i-j)} \right)
\end{equation}
It remains to collect signs
and powers of $q$ in order to prove:

\begin{theorem}
  Let $\mathrm{gcd}(k,n)=1$.  Let $C$ be the curve cut out by $y^k = x^n$ 
  and let $p$ be the origin;  
  $\mu = (k-1)(n-1)$ is the Milnor number of this singularity, its
  link is the $k,n$ torus knot, and
  \begin{equation*}
    \JJ(\mbox{ k,n torus knot })  =   
    (a/q)^{\mu} (1-q^2) \!\! \int\limits_{\Hilbert{{C_p}}}\!\!
    q^{2l} (1-a^2)^{m-1} \, d \chi
  \end{equation*}
\end{theorem}

\begin{corollary}
  Let $C$ be a rational curve, smooth away from a point $p$, and 
  formally isomorphic at 
  $p$ to $\mathrm{Spec}\,\, \CC[[x,y]]/(y^k=x^n)$.  Write 
  ${b \choose c}_{q^2}$ for $\frac{[b]!}{[c]![b-c]!}$.  Then
  \[
  (1-q^2)^{2}\int\limits_{\Hilbert{C}} \! q^{2l} \, d\chi \,\,\, = 
  \frac{{k+n \choose k}_{q^2}}{{k+n \choose 1}_{q^2}}
  \]  
\end{corollary}
\begin{proof}
  We have proven Conjecture \ref{conj:homfly} in the case of the
  singularity in question, which implies Conjecture \ref{conj:justqs}.
  Substituting in, we see
  \begin{eqnarray*}
    (1-q^2)^{2}\int\limits_{\Hilbert{C}} \! q^{2l} \, d\chi \,\,\, & = &
    \frac{(1-q^2)}{(1-q^{2k})}
    \sum_{j=0}^{k-1} (-1)^j \frac{(q^2)^{j n + (k-1-j)(k-j)/2}}{[j]!\,[k-1-j]!}  
    \prod_{i=j+1-k}^jq^{2i}
    \\
    & = & \frac{(1-q^2)}{[k]!}
    \sum_{j = 0}^{k-1} {k-1 \choose j}_{q^2} q^{j(j-1)} (-q^{2(n+1)})^j \\
  \end{eqnarray*}
  Now
  we use the ``Newton formula for Gaussian binomials''
  \[
    \sum_{r=0}^s q^{r(r-1)} {s \choose r}_{q^2} t^r = \prod_{r=0}^{s-1}(1+q^{2r} t) 
  \]
  to deduce 
  \[    (1-q^2)^{2}\int\limits_{\Hilbert{C}} \! q^{2l} \, d\chi \,\,\, =
  \frac{(1-q^2)}{[k]!} \prod_{j=0}^{k-2} (1-(q^2)^{n+j+1}) = 
  \frac{{k+n \choose k}_{q^2}}{{k+n \choose 1}_{q^2}}
  \]
\end{proof}

\noindent Setting $q=1$ recovers Beauville's  formula \cite{B} for the Euler
  number of the compactified Jacobian.

\section{Hilbert schemes of unibranch singularities}
\label{sec:Hilbert}

Let $C$ be the germ of a unibranch singularity, and  
fix a normalization $\oO = \oO_C \hookrightarrow \CC[[t]]$.  
Let $\nu:\CC[[t]] \to \NN$ 
be the valuation taking a series to the degree of its lowest degree term.  
Lemma \ref{lem:lengths} shows that $\nu(f) = \dim_\CC \oO/f\oO$, so
$\nu|_{\oO}$ is independent of the choice of normalization.
Let $\Gamma = \nu(\oO)$ -- when appropriate we implicitly exclude
$0$ from the domain of $\nu$ -- 
then $\Gamma$ is a cofinite\footnote{If not, then the fraction field of $\oO$ would be 
some $k((t^r)) \subsetneq k((t))$, but $\oO$ and
its normalization $k[[t]]$ must share the same fraction field.  For
the theory of unibranch singularities, we refer
to the book of Zariski and Teissier \cite{ZT}.} 
subset of $\NN$ closed
under addition.  We
employ the filtration $F_k \oO = \{f \in \oO | \nu(f) \ge k\}$.

\begin{lemma}
  $\dim F_k \oO / F_{k+1} \oO \le 1$. 
\end{lemma}

If $J$ an ideal of $\oO$, then $\nu(J) \subset \NN$ 
is a semigroup ideal: $\nu(J) + \Gamma \subset \nu(J)$.

\begin{corollary} \label{cor:semiidealcolength}
  Let $J$ be an ideal of $\oO$.  Then
  $\dim_\CC \oO/J = \# \nu(\oO) \setminus \nu(J)$. 
\end{corollary}
\begin{proof}
  As $\dim_\CC \oO/J$ is finite,
  $\dim \oO/J = \sum \dim F_n \oO / (F_n\oO \cap J + F_{n+1} \oO)$.  This contributes
  $1$ exactly when there is a ring element with valuation $n$, but no ideal
  element with valuation $n$; i.e., for $n \in \nu(\oO) \setminus \nu(J)$. 
\end{proof}

For elements 
$a_0, \ldots, a_k \in \Gamma$, we denote the semigroup ideal they
generate by 
\[(a_0,\ldots,a_k)_\Gamma = \{a_i + \gamma_i | \gamma_i \in \Gamma \}\]

\begin{corollary} 
  Let $J$ be an ideal of $\oO$.  For
  $f_0, \ldots, f_k \in J$, 
  \[(\nu(f_0),\ldots,\nu(f_k))_\Gamma = \nu(J)  \implies
  (f_0,\ldots,f_k) = J\]
\end{corollary}
\begin{proof}
  Let $J'$ be the ideal generated by the $f_i$.   Since 
  $J' \subset J$, surely $\nu(J') \subset \nu(J)$.  But 
  \[\nu(J') \supset (\nu(f_0),\ldots,\nu(f_k))_\Gamma = \nu(J)\]  
  Thus $    
    \dim J / J' = \dim \oO / J' - \dim \oO / J  = 
    \# \Gamma \setminus \nu(J')  - 
    \# \Gamma \setminus \nu(J) = 0$.
\end{proof}

\noindent {\bf Remark.}  The converse is false.  
The ring $\oO = k[[t^4, t^6+ t^7]]$ has semigroup
$\langle 4,6,13\rangle$.  Its maximal ideal 
$M = (t^4, t^6+t^7)$ has semigroup ideal $\nu(M) = (4,6,13)_\Gamma$.  In fact, 
there is no ideal $J$ with $\nu(J) = (4,6)_\Gamma$:  any
such ideal contains $t^4, t^6 + t^7$, hence
$(t^6 + t^7)^2 - (t^4)^3 = 2 t^{13} + t^{14}$.
\vspace{2mm}

We write $\Hilbert{\Gamma}$ for the set of semigroup ideals of
$\Gamma$, and view $\nu$ as a constructible map
\[\nu:\Hilbert{C} \to \Hilbert{\Gamma}\]
which lets us define a 
stratification $C^{[\mathfrak{j}]} = \nu^{-1}(\mathfrak{j})$.  Corollary
\ref{cor:semiidealcolength} shows that this is a sub-stratification of
the usual one by length: $C^{[\mathfrak{j}]} \subset C^{[\#\Gamma \setminus \mathfrak{j}]}$. 

\vspace{1mm}
\noindent {\bf Remark.}  Teissier \cite{ZT} constructs a
$\CC^*$-equivariant deformation 
$\mathcal{C} \to \AA^1$ whose generic fibre is $C$ and 
whose special fibre is 
the not-necessarily-planar $C_\Gamma = \mathrm{Spec}\,\,\CC[[\Gamma]]$.   The
central fibre carries a
natural $\CC^*$ action.  Lifting the action to the Hilbert scheme, 
$(\Hilbert{{C_\Gamma}})^{\CC^*} = \Hilbert{\Gamma}$.  The map $\nu$ amounts 
to taking the $t \to 0$ limit of the $\CC^*$ action on the relative
Hilbert scheme of $\mathcal{C}/\AA^1$.  
\vspace{1mm}

\begin{lemma}
  For $g \in \Gamma$ there exists a unique element
  $\tau_g \in \oO \subset \CC[[t]]$  
  of the form 
  \[\tau_g = t^g + \sum_{g < i \notin \Gamma} c_i t^i\]
\end{lemma}
\begin{proof}
  Since $g \in \Gamma$, there is some element $\tilde{\tau} \in \oO$
  with leading term $t^g$.  Terms in $\tilde{\tau}$ of degrees
  in $\Gamma$ may be successively removed; the process converges
  since $\oO$ is complete.  Given $\tau_g \ne \tau_g'$ of the form
  prescribed, we have $\nu(\tau_g -\tau_g') \ne \Gamma$, which is absurd.
\end{proof}

\begin{definition} \label{def:stratification}
  Fix $a_0,\ldots,a_k \in \Gamma$.
  Use $\Sigma_\lambda := \Gamma_{> a_\lambda} \setminus (a_0,\ldots,a_k)_\Gamma$
  to index the set of indeterminates
  $\mathbf{S} := \{S_{\lambda,i}\,|\,i\in \Sigma_\lambda \}$.  
  Let 
  $V_{a_0,\ldots,a_k} := \mathrm{Spec}\,\, \CC[\mathbf{S}]$, and 
  define $\mathbf{f}_\lambda \in \oO [\mathbf{S}]$ by
  \[
  \mathbf{f}_\lambda := \tau_{a_\lambda} + \sum_{i \in \Sigma_\lambda} 
  \tau_i S_{\lambda, i} 
  \]
  and form the ideal $\mathcal{J} :=  (\mathbf{f}_0, \ldots, \mathbf{f}_k)
  \subset \oO [\mathbf{S}]$.
  Consider the subvariety
  \[ U_{a_0,\ldots,a_k} := \{s \in V_{a_0,\ldots,a_k}\,|\,  
  \nu(\mathcal{J}|_s) = (a_0,\ldots,a_k)_\Gamma\} \]
  As we always have $(a_0,\ldots,a_k)_\Gamma \subset \nu(\mathcal{J}|_s)$,
  $U_{a_0,\ldots,a_k}$ is equivalently described as the locus where
  $\dim \oO/\mathcal{J}_s = \# \Gamma \setminus (a_0,\ldots,a_k)_\Gamma$.
  As the Hilbert polynomial is constant and the base is reduced, 
  $\oO [\mathbf{S}]/ \mathcal{J}$ 
  is flat over $U_{a_0,\ldots,a_k}$.  Thus we get a map
  \[\Psi_{a_0,\ldots,a_k}:U_{a_0,\ldots,a_k} \to 
  C^{[(a_0,\ldots,a_k)_\Gamma]} \subset \Hilbert{C}\]
\end{definition}

\begin{theorem} \label{thm:stratification}
  $\Psi_{a_0,\ldots,a_k}:U_{a_0,\ldots,a_k} \to 
  C^{[(a_0,\ldots,a_k)_\Gamma]}$  is bijective. 
\end{theorem}
\begin{proof}  
  Let $\mathfrak{j} = (a_0,\ldots,a_k)_\Gamma$. 
  Consider a closed point $(s_{\lambda,i})$ 
  in the preimage of $C^{[\mathfrak{j}]}$.  
  This corresponds to an ideal $J = (f_0, \ldots, f_k)$ where 
  $f_i = \tau_{a_\lambda} + \sum_{i \in \Sigma_\lambda} \tau_i s_{\lambda, i}$.  Suppose
  $J$ is also the ideal corresponding to $(s'_{\lambda,i\in\Sigma_\lambda})$, hence
  has generators 
  $f'_\lambda = \tau_{a_\lambda} + \sum_{i \in \Sigma_\lambda} \tau_i s'_{\lambda, i}$.  Now
  $f_\lambda - f'_\lambda \in J$, but on the other hand 
  $\nu(f_\lambda - f'_\lambda) \in \Gamma \setminus \mathfrak{j}$ unless
  $f_\lambda = f'_\lambda$.  Thus the map is injective.  
  
  For surjectivity, fix an ideal $J$ with $\nu(J) = \mathfrak{j}$. 
  Choose lifts of $f_\lambda \in J$ of the $a_\lambda$.  A generating
  set will still generate if we modify 
  $f_\lambda \to u f_\lambda + \sum_{\nu \ne \lambda} v f_\nu$ for invertible
  $u$ and arbitrary $v$.  
  Iteratively removing terms of the form $t^n$ for 
  $n \in \mathfrak{j} \setminus \lambda$ from $f_\lambda$ converges
  to yield generators of $J$ of the form required by Definition
  \ref{def:stratification}.
\end{proof}

\noindent{\bf Remark.}  As defined, the $V_{a_0,\ldots,a_k}$ depend on the
choice of generators of the semigroup ideal $(a_0,\ldots,a_k)_\Gamma$.  
However, a semigroup ideal has a unique minimal generating set; henceforth
if we write $V_{\mathfrak{j}}$ to mean that the minimal generating set
of the semigroup ideal is chosen.  On the other hand, while it may likewise
seem that $U_{a_0,\ldots,a_k}$ depends on the choice of generators, Theorem
\ref{thm:stratification} implies that
all choices yield spaces which biject onto $C^{[\mathfrak{j}]}$. 
\vspace{2mm}

\noindent {\bf Caution.} 
The function giving the number of generators need not be constant on
the $U_{\mathfrak{j}}$.   

\vspace{2mm}
For computations, the following consequence of Lemma \ref{lem:lengths} 
is useful. 

\begin{corollary} \label{cor:lengths}
  Let $\Gamma \subset \NN$ be a semigroup with 
  $\#\NN \setminus \Gamma < \infty$.   For $i \in \Gamma$,
  $\# \Gamma \setminus (i + \Gamma)  = i$. 
  More generally, let $\Gamma \subset \Delta \subset \NN$, 
  and suppose $i + \Delta \subset \Gamma$. 
  Then
  \[\# \Gamma \setminus (i + \Delta) = i - \# \Delta \setminus \Gamma\]
\end{corollary}
\begin{proof}
  Consider $\CC[\Gamma]$, 
  the ring with generators $\{x^\gamma\}_{\gamma \in \Gamma}$
  and relations $x^{\gamma_1}  x^{\gamma_2} = x^{\gamma_1 + \gamma_2}$.  Now apply
  Lemma \ref{lem:lengths} to $\CC[\Gamma] \hookrightarrow \CC[x]$ and
  $x^i \in \CC[\Gamma]$ to see   $\# \Gamma \setminus (i + \Gamma)  = i$. 
  The final statement follows from
  $\# \Gamma \setminus (i + \Delta) + \# (i + \Delta) \setminus (i + \Gamma) = 
  \# \Gamma \setminus (i + \Gamma) = i$ and
  $\# (i + \Delta) \setminus (i + \Gamma) = 
  \# \Delta \setminus \Gamma$.
\end{proof}

\section{The singularity with semigroup $\langle 4,6,13 \rangle$}
\label{sec:4613}

We consider now the ring $\oO = \CC[[t^4, t^6 + t^7]]$ 
and the singularity $C = \mathrm{Spec}\,\oO$.  
Let us calculate the semigroup. As
$(t^6 + t^{7})^2 - (t^4)^3 = t^{13} (2 + t)$ we see 
$4, 6, 13 \in \nu(\oO)$.  
Suppose there is
$P(x,y) \in \CC[[x,y]]$ such that $P(t) = P(t^4,t^6 + t^{7})$ has leading
term $t^{15}$.  Then certainly $P(x,y)$ must have two monomials, 
$x^ay^b$ and $x^c y^d$, such that $4a + 6b = 4c + 6d < 15$; moreover
their leading terms must cancel when evaluated at $x = t^4$,  $y=t^6 + t^7$.
By inspection, the first condition is only satisfied for 
$4a + 6b = 12$, but in this case we have already seen that 
the leading term of $P(t)$ is $t^{13}$.  So $15 \notin \nu(\oO)$, and 
the semigroup is
\[\Gamma = \langle 4, 6, 13 \rangle = \{0,4,6,8,10,12,13,14,16,17,\ldots\}\]
In fact, Zariski has shown that this is the only plane singularity with
this semigroup \cite{ZT}.  

The link of this singularity is 
the (2,13) cable of the (2,3) torus knot \cite{EN}.  Its HOMFLY polynomial,
as calculated by computer, is
\begin{eqnarray*}
- & a^{22} & ( 3 + 4z^2 + z^4 )\\
+ & a^{20} & (20 + 70 z^2 + 84 z^4 + 45 z^6 + 11 z^8 + z^{10}) \\
- & a^{18} & ( 39 + 220 z^2 + 468 z^4 + 496 z^6 + 286 z^8 + 91 z^{10} + 15z^{12} + z^{14}) \\
+ & a^{16} & (23 + 179 z^2 + 540 z^4 + 836 z^6 + 726 z^8 + 365 z^{10}
+ 105 z^{12} + 16 z^{14} + z^{16})
\end{eqnarray*}
where $z = q-q^{-1}$.  
According to Conjecture \ref{conj:justqs}, the coefficient
of $z^{2h}$ in the bottom row above is the number $n_h$ of Pandharipande
and Thomas \cite[Appendix 2]{PT3}. In particular, the Euler characteristic of
the Jacobian factor of this singularity should be $n_0 = 23$.  This was previously
calculated by Piontkowski \cite{Pi} using similar methods.\footnote{
Indeed, Piontkowski also
determines the Euler characteristic of the Jacobian factor for all
singularities with semigroups $\langle 4, 2q, s \rangle$, 
$\langle  6, 8, s \rangle$, and $\langle 6, 10, s \rangle$.  He does so
by constructing a stratification of the Jacobian factor by affine
spaces, and suggests that it is unlikely that any other singularities
will admit such a stratification.  We remark that
his list exhausts the algebraic $(2,n)$ cables of $(2,p)$, 
$(3,4)$ and $(3,5)$ torus knots, or in other words, the $(2,n)$ cables of
knots arising from simple singularities.  
}

We turn now to the calculation of the integral in Conjecture \ref{conj:homfly}
by using the stratification of Section \ref{sec:Hilbert}.  
Moreover we group the semigroup ideals which are isomorphic as 
$\Gamma$-modules, i.e., differ merely by a shift.  Let
$\mathrm{Mod}(\Gamma)$ denote the set of $\Gamma$ submodules of $\NN$ 
containing zero. 
Using Theorem \ref{thm:stratification} and Corollary \ref{cor:lengths}, 
\begin{equation} \label{eq:breakdown}
  \int\limits_{\Hilbert{{C}}}\!\!
    q^{2l} (1-a^2)^{m-1} \, d \chi \,\,\, = \!\!\!
    \sum_{\Delta \in \mathrm{Mod}(\Gamma)} \!\!\! q^{-(8 - \#\NN \setminus \Delta)} 
    \sum_{i + \Delta \subset \Gamma} 
    q^{2i} \int \limits_{U_{i+\Delta}}  (1-a^2)^{m-1} \, d \chi
\end{equation}

The combinatorial data required to compute the right hand side
is tabulated in Figure \ref{fig:data}, which appears at the end
of the article.  We compute the rightmost
integral by determining the spaces
$U_{i+\Delta}$, together with their stratifications by the number of generators.

\begin{lemma} \label{lem:whojumps}
  Let $\Gamma = \langle 4, 6, 13 \rangle$, and consider
  $\Delta \in \mathrm{Mod}(\Gamma)$.
  Let $0 = \alpha_0 < \alpha_1 < \ldots$ be its minimal set of generators.  
  Choose arbitrary
  $f_0, f_1, \ldots \in k[[t]]$ with degrees $\alpha_0, \alpha_1, \ldots$.  
  Then if $\eta = \nu (\sum f_i \phi_i) \notin \Delta$ for $\phi_i \in \oO$, 
  then:
  \begin{itemize}
  \item $\alpha_1 = 2$ and 
    $1, 3 \notin \Delta$
  \item $\eta \in \{ 7,9,11,15\}$ 
  \item $\nu (f_0 \phi_0) = \nu( f_1 \phi_1) = \min\{ \nu( f_i \phi_i)\}$ 
  \end{itemize}
\end{lemma}
\begin{proof}
  Assume there exist $\phi_i \in \oO$ such that 
  $\eta = \nu(\sum f_i \phi_i) \notin \Delta$.
  In particular, we must have 
  $\nu(\sum f_i \phi_i) > \min \, \nu(f_i \phi_i)$, thus the lowest
  degree terms must cancel, thus there must be at least two of them.
  Say they are $f_j$ and $f_k$; let $\alpha_j < \alpha_k$.  We have
  \[ \eta > \alpha_k + \nu(\phi_k)  = \alpha_j + \nu(\phi_j)\]  

  We cannot
  have $\nu(\phi_k)=0$ since $\alpha_k$ is a necessary generator; thus
  $\nu(\phi_k) \ge 4$.  Since $\nu(\phi_j) > \nu(\phi_k)$ we also have
  $\nu(\phi_j) \ge 6$.  
  As $\eta > \alpha_j + \nu (\phi_j) \ge 6$ and $\eta \notin \Delta$, we must
  have $\eta \in \{7,9,11,15\}$.  Since $\eta\notin \Delta$, no odd
  number less than $\eta - 2$ can be in $\Delta$.  This implies that
  $1,3 \notin \Delta$ and $a_j, a_k$ are even.  This can only happen
  if $a_j = 0$ and $a_k = 2$. 
\end{proof}

\begin{lemma} \label{lem:sevenornine}
  Given two series
  \begin{eqnarray*}
    f_0 & = & 1+a_1 t + a_3 t^3 + \ldots \\
    f_1 & = & t^2(1+b_1 t + b_3 t^3 + \ldots) \\
  \end{eqnarray*}
  we see that  
  \begin{eqnarray*}
    \deg_t ( (t^6 + t^7)f_0 - t^4 f_1) \ge 7 & \text{with equality unless}
    &  b_1 - a_1 = 1 \\
    \deg_t((t^6 + t^7)f_1 - t^8 f_0) \ge 9 & \text{with equality unless} &  
    a_1 - b_1 = 1 \\
  \end{eqnarray*}
  The equations cannot hold simultaneously, so at least
  one of the series has the specified degree. 
\end{lemma}

\begin{corollary} \label{cor:list}
  Let $0 \in \Delta \subset \NN$ be a $\Gamma$-module with minimal
  generators $\alpha_0, \alpha_1, \ldots$.  Choose lifts 
  $f_0, f_1, \ldots \in \CC[[t]]$.  Let $\Delta' = \nu((f_0,f_1,\ldots))$. 
  Then if $\Delta \ne \Delta'$, then $\Delta \to \Delta'$ appears on
  the following list. 
  \begin{itemize}
    \item $(0,2) \to (0,2,7), (0,2,9), (0,2,7,9)$ 
    \item $(0,2, 5) \to (0,2,5,7)$  
    \item $(0,2, 7) \to (0,2,7,9)$
    \item $(0,2, 9) \to (0,2,7,9), (0,2,9,11)$
    \item $(0,2, 11) \to (0,2,7), (0,2,7,9), (0,2,9,11)$
  \end{itemize}  
  For all modules $\Phi$ not occuring on the list, 
  \[\int \limits_{U_{i+\Phi}}  (1-a^2)^{m-1} \, d \chi = 
  (1-a^2)^{m(\Phi)-1}\]
  where $m(\Phi)$ is the number of generators of $\Phi$ as a
  $\Gamma$-module.  
  Moreover, if $\Phi = (0,2)$ or $(0,2,11)$, then $U_{i+\Phi} = \emptyset$, so
  \[\int \limits_{U_{i+\Phi}}  (1-a^2)^{m-1} \, d \chi =  0\]  
\end{corollary}
\begin{proof}
  All $\Gamma$ modules are listed in Figure \ref{fig:data}.  Checking
  the criterion of Lemma \ref{lem:whojumps} on each of them yields
  the list of possible $\Delta$.
  The list of possible $\Delta'$ comes from the criterion of 
  Lemma \ref{lem:sevenornine}:
  if $0$ and $2$ are in $\Delta$, then $7$ or $9$ is in $\Delta'$. 
  If a semigroup module $M$
  never occurs as one of the $\Delta$ above, 
  Theorem \ref{thm:stratification} implies that $U_{i+M}$ is an affine space.
  If $M$ never occurs as one of the $\Delta'$, then Theorem
  \ref{thm:stratification} implies that $\nu(I) = i + M \implies
  m(I) = m(M)$.
  The final statement of the corollary is immediate
  from Lemma \ref{lem:sevenornine}. 
\end{proof}

We proceed to analyse the remaining modules.  
Suppose $\Delta \in \mathrm{Mod}(\Gamma)$, 
$0,2 \in \Delta$, and $1,3 \notin \Delta$.  
Fix $i$ such that $i + \Delta \subset \Gamma$, and let 
$V_{i+\Delta}$ be the affine space of Definition \ref{def:stratification}.  
$V_{i+\Delta}$ has coordinate functions $\mathbf{a},\mathbf{b}$ 
giving respectively the 
coefficient of $x^{i+1}$ of the generator of degree $i$, and 
the coefficient of $x^{i+3}$ in the generator of degree $i+2$.  
Let $I$ be an ideal corresponding to some point in $V_{i + \Delta}$ such
that $a = \mathbf{a}(I)$ and $b = \mathbf{b}(I)$.   
Lemma \ref{lem:sevenornine} implies that $i+7 \in \nu(I)$ 
unless $b - a = 1$, and $i+9 \in \nu(I)$ unless
$a-b = 1$.  In fact, $b - a$ is 
identically 1 if and only if $i + 7 \notin \Gamma$,
identically $-1$ if and only if $i + 9 \notin \Gamma$, and otherwise
may assume any value.  This may be seen from the explicit form of a general
element of $\oO$:
\[c_0 + c_4 t^4 + c_6 (t^6+t^7) 
+ c_8 t^8 + c_{10}(t^{10} + t^{11}) 
+ c_{12}t^{12} + c_{13} t^{13} + c_{14}(t^{14} + t^{15}) + 
\sum_{n \ge 16} c_n t^n\]

We will now determine the $U_\mathfrak{j}$ of Definition \ref{def:stratification}
by computing their complements in the affine spaces $V_{\mathfrak{j}}$.  Recall
that the space $V_\mathfrak{j}$ depended on a choice of generators of 
$\mathfrak{j}$; we use the set of generators indicated (which is always the
minimal set of generators).   For $I$ an ideal of $\oO$, we say its
{\em type} is the semigroup ideal $\nu(I) \subset \Gamma$.

\begin{itemize}
  \item The 
    complement of $U_{i+(0,2,5)}$ inside $V_{i+(0,2,5)}$ will be the
    ideals whose type is $i + (0,2,5,7)$.  As $i+9 \in \Gamma$ since
    $i + 5 \in \Gamma$, the $i+7$ appears in the complement of a hyperplane
    if at all.  Thus $U_\Gamma$ has Euler characteristic 1.  On the other
    hand, we see in Corollary \ref{cor:list} that an ideal of type
    $i + (0,2,5)$ always has three generators.  Thus
    \[\int \limits_{U_{i+(0,2,5)}} \!\!\!\! (1-a^2)^{m-1} \, d \chi
    = (1-a^2)^2\]
  \item The complement of $U_{i+(0,2,7)}$ in $V_{i+(0,2,7)}$ consists of 
    semigroup ideals whose type is $i + (0,2,7,9)$.  Since $i+7$ is
    in the semigroup ideal, $i+9$ fails to be in the semigroup ideal 
    either on a hyperplane
    or in all of $U_{i+(0,2,7)}$.  Moreover when $i+9$ fails to be in
    the semigroup ideal, by Lemma \ref{lem:sevenornine}, the
    generators of degree $i, i+2$ generate the ideal.  Thus 
    \[\int \limits_{U_{i+(0,2,5)}} \!\!\!\! (1-a^2)^{m-1} \, d \chi
    = (1-a^2)\]
  \item $U_{i+(0,2,5,7)} = V_{i+(0,2,5,7)}$.  Since $i+9$ is in
    the semigroup ideal, the space on which
    the generator of degree seven is not needed is the complement
    of a hyperplane and thus has Euler characteristic zero.  Thus:
    \[\int \limits_{U_{i+(0,2,5,7)}} \!\!\!\!\!\! (1-a^2)^{m-1} \, d \chi
    = (1-a^2)^3\]
  \item $U_{i+(0,2,7,9)} = V_{i+(0,2,7,9)}$.  Since $7$ and $9$ are both
    in the semigroup ideal, the generator of degree $i+7$ is unnecessary
    in the complement of a hyperplane, and the generator of degree
    $i+9$ is unnecessary in the complement of a parallel hyperplane. 
    \[\int \limits_{U_{i+(0,2,7,9)}} \!\!\!\!\!\! (1-a^2)^{m-1} \, d \chi
    = 2(1-a^2)^2 - (1-a^2) \]
  \item The complement of $U_{i+(0,2,9)}$ in $V_{i+(0,2,9)}$ consists of the
    locus where the type is $(0,2,7,9)$ or $(0,2,9,11)$.  We already 
    understand that the first happens, if at all, on the the complement
    of a hyperplane.  Let $U'_{i+0,2,9}$ be the affine space on which the type
    is not $(0,2,7,9)$.  Here the generators may be written:
    \begin{eqnarray*}
      f & = & t^i + at^{i+1} + a_3 t^{i+3} + a_5 t^{i+5} + a_7 t^{i+7} + 
      a_{11} t^{i+11} \\
      g & = & t^{i+2} + (a+1) t^{i+3} + b_3 t^{i+5} + b_5 t^{i+7} + 
      b_9 t^{i+11} \\
      h & = & t^{i+9} + c t^{i+11}
    \end{eqnarray*}
    Writing $x=t^4$ and $y=t^6+t^7$, we see that 
    \[(x^2+xy(a+1))f - yg + 2h = 
    ((a+1)^2 + 2c + a_3 - b_3)t^{11} + O(t^{12})\]
    Now let $U''_{i+(0,2,9)}$ be the locus where
    the above coefficient of $t^{11}$ vanishes; since it is given
    as $c = ((b_3 - a_3) + (a+1)^2)/2$, it is isomorphic to affine
    space.  Restricting to $U''_{i+(0,2,9)}$, one checks that the
    only term of degree 11, modulo $t^{12}$, is 
    \[(2x - (a^2-a_3+b_3)y)g + 
    (2(ax^2+a_3 x y -y)+(a^2-a_3+b_3)(x^2+xy(a+1)))f\]
    and the coefficient of $t^{11}$ is 
    $C(a,a_3,a_5,b_3,b_5) = 2b_5-2a_5 + P(a,a_3,b_3))$, 
    where $P$ is some polynomial.  Evidently $U_{i+(0,2,9)}$ is
    the locus inside $U''_{i+(0,2,9)}$ where $C$ vanishes.  If
    either $i+5$ or $i+7$ is in $\Gamma$, then $C$ vanishes on
    a hypersurface isomorphic to affine space.  In fact, by
    inspection of $\Gamma$, one sees that
    whenever $i+(0,2,9,11) \subset \Gamma$,
    either $i+5$ or $i+7$ is in $\Gamma$.  In the case
    $i+(0,2,9,11) \not\subset \Gamma$ and hence $i+11 \notin \Gamma$, 
    then $C$ vanishes identically, $U_{i+(0,2,9)} = U''_{i+(0,2,9)}$, 
    which we already knew was an affine space.  In any event 
    $U_{i + (0,2,9)}$ has Euler characteristic
    $1$.  Finally 
    by Lemma \ref{lem:sevenornine} the generator of degree
    $i+9$ is superfluous, so: 
    \[\int \limits_{U_{i+(0,2,9)}} \!\!\!\! (1-a^2)^{m-1} \, d \chi
    = (1-a^2) \]
  \item The complement of $U_{i+(0,2,9,11)}$ inside $V_{i+(0,2,9,11)}$
    is the locus where the semigroup type is $i+(0,2,7,9)$.  We have
    seen that the $i+7$ appears in the complement of a hyperplane,
    if at all.  Thus $U_{i+(0,2,9,11)}$ is isomorphic to affine space.
    We have also seen that the generator of degree $i+9$ is always
    superfluous;  by the argument given for $(0,2,9)$, 
    the generator of degree $i+11$ is superfluous
    in the complement of an affine space.  Thus: 
    \[\int \limits_{U_{i+(0,2,9,11)}} \!\!\!\!\!\! (1-a^2)^{m-1} \, d \chi
    = (1-a^2)^2 \]
\end{itemize}

This completes the determination of the integrals appearing
in Equation (\ref{eq:breakdown}).  Summing the contributions yields
complete agreement with the HOMFLY polynomial.

\begin{figure}[lh]
\[
\begin{array}{|c||c|r|}
\hline
\Delta \in \mathrm{Mod}(\Gamma) & 8 - \#\NN \setminus \Delta & 
i| i+ \Delta \subset \Gamma \\ 
\hline
(0) & 0 & 0,4,6,8,10,12,13,14,16+ \\
\hline
(0,1) & 6 &  12, 13, 16+ \\
(0,3) & 5 & 10, \phantom{12,} 13, 14, 16+ \\
(0,5) & 4 & 8,\phantom{10,} 12, 13, 14, 16+ \\
(0,7) & 3 & 6, \phantom{8,} 10, 12, 13, 14, 16+ \\
(0,9) & 2 & 4, \phantom{6,} 8, 10, 12, 13, 14, 16+ \\
(0,11) & 2 & 6,8,10, 12, 13, 14, 16+ \\
(0,15) & 1 & 4,6,8,10,12,13,14,16+ \\
\hline
(0,2) & 2  & 4,6,8,10,12,\phantom{13,}14,16+ \\
\hline
(0,1,3) & 7 & 13,16+ \\
(0,3,5) & 6 & 13,14,16+ \\
(0,5,7) & 5 & 12,13,14,16+ \\
(0,7,9) & 4 & 10,12,13,14,16+ \\
(0,9,11) & 3 & 8,10,12,13,14,16+ \\
\hline
(0,1,2) & 7 & 12,\phantom{13,14,}16+ \\
(0,2,3) & 6 & 10,\phantom{12,13,}14,16+ \\
(0,2,5) & 5 & 8, \phantom{10,} 12, \phantom{13,}14, 16+ \\
(0,2,7) & 4 & 6,\phantom{8,} 10, 12, \phantom{13,}14, 16+ \\
(0,2,9) & 3 & 4,\phantom{6,} 8, 10, 12, \phantom{13,}14, 16+ \\
\hline
(0,2,11) & 3 &  6,8,10,12,\phantom{13,}14,16+ \\
\hline
(0,1,2,3) & 8 & 16+ \\
(0,2,3,5) & 7 & 14,16+ \\
(0,2,5,7) & 6 & 12, \phantom{13,}14, 16+ \\
(0,2,7,9) & 5 & 10, 12, \phantom{13,}14, 16+ \\
(0,2,9,11) & 4 & 8, 10, 12, \phantom{13,} 14, 16+ \\
\hline
\end{array}
\]
\caption{ \label{fig:data} The combinatorial data required by Equation (\ref{eq:breakdown}).}
\end{figure}

\vspace{+20pt}
\noindent
Alexei Oblomkov \\
Department of Mathematics \\
University of Massachusetts, Amherst \\
Amherst, MA 01003 \\
{\tt oblomkov@math.umass.edu} \\

\vspace{+10pt}
\noindent Vivek Shende \\
Department of Mathematics \\
Massachusetts Institute of Technology \\
Cambridge, MA 02139 \\
{\tt vivek@math.mit.edu} \\

\end{document}